\documentclass[mn,a4paper,fleqn%
]{w-art}
\usepackage{times,cite,w-thm}
\usepackage{amssymb}
\theoremstyle{plain}

\theoremstyle{definition}

\usepackage[]{graphicx}
\numberwithin{equation}{section}

\newcommand{\rn}{{\mathbb R}^n}
\newcommand{\rnp}{{\mathbb R}^n_+}

\newcommand{\comega}{\overline\Omega }

\newcommand{\simto}{\overset\sim\rightarrow}

\begin{document}
\DOIsuffix{theDOIsuffix}
\Volume{248}
\Month{01}
\Year{2007}
\pagespan{1}{}
\Receiveddate{XXXX}
\Reviseddate{XXXX}
\Accepteddate{XXXX}
\Dateposted{XXXX}



\title[ Regularity of spectral problems ]{Regularity of spectral fractional Dirichlet and Neumann
problems}



\author[G. Grubb]{Gerd Grubb
\footnote{Corresponding author\quad E-mail:~\textsf{grubb@math.ku.dk}, Phone +45\, 3532\,0743  Fax  +45\, 3532\,0704 .}}
\address
{Department of Mathematical Sciences,
University of Copenhagen, 
Universitetsparken 5,
DK-2100 K\o{}benhavn,  
Denmark}

\begin{abstract}

Consider the fractional powers $(A_{\operatorname{Dir}})^a$ and
$(A_{\operatorname{Neu}})^a$ of the Dirichlet and  Neumann
realizations  of a second-order strongly elliptic differential
operator $A$ on a smooth bounded subset $\Omega $ of
${\mathbb R}^n$.  Recalling the results on complex powers and complex
interpolation of domains of elliptic boundary value problems by Seeley
in the 1970's, we demonstrate how they imply regularity properties in
full scales of  $H^s_p$-Sobolev spaces and
H\"older spaces, for the solutions  of the
associated equations.
Extensions to nonsmooth situations for low values of $s$ are derived
by use of recent results on $H^\infty $-calculus.
We also include an overview of the various Dirichlet- and
Neumann-type boundary problems associated with the fractional Laplacian.

\end{abstract}
\maketitle                   






\section{Introduction}\label{secIntro}

There is currently a great interest in fractional powers of the
Laplacian  $(-\Delta )^a$
 on ${\mathbb R}^n$, $a>0$,  and derived operators associated with a
 subset of ${\mathbb R}^n$. The fractional Laplacian $(-\Delta )^a$ can be described as the pseudodifferential
operator 
\begin{equation}u\mapsto (-\Delta )^au=\mathcal F^{-1}(|\xi |^{2a}\hat u(\xi
))=\operatorname{Op}(|\xi |^{2a})u,\label{1.1}
\end{equation}
with symbol $|\xi |^{2a}$, see also \eqref{6.1} below. Let $\Omega $ be a bounded $C^\infty $-smooth subset of
${\mathbb R}^n$. Since $(-\Delta )^a$ is nonlocal, it is not obvious how
to define boundary value problems for it on $\Omega $, and in fact
there are several interesting choices.

One choice for a Dirichlet realization on $\Omega $ is to take the power $(-\Delta _{\operatorname{Dir}})^a$ defined from the Dirichlet
realization $-\Delta _{\operatorname{Dir}}$ of $-\Delta $ by spectral theory in
the Hilbert space $L_2(\Omega )$; let us 
call it ``the spectral Dirichlet
fractional Laplacian'', following a suggestion of Bonforte, Sire and
Vazquez \cite{BSV14}.

Another very natural choice is to take   the
Friedrichs extension of the operator $r^+(-\Delta )^a|_{C_0^\infty
(\Omega )}$ (where $r^+$
denotes restriction to $\Omega $); let us denote it $(-\Delta
)^a_{\operatorname{Dir}}$ and call it ``the restricted Dirichlet fractional
Laplacian'', following \cite{BSV14}. 

Both choices enter in nonlinear PDE;  $(-\Delta
)^a_{\operatorname{Dir}}$ is moreover important in probability theory.
 The operator $-\Delta $
can be replaced by a variable-coefficient strongly elliptic second-order
operator $A$ (not necessarily symmetric).

For the restricted Dirichlet fractional Laplacian, detailed regularity properties of
solutions of  $(-\Delta
)^a_{\operatorname{Dir}}u=f$ in H\"older spaces and $H^s_p$ Sobolev spaces have just recently been shown, in Ros-Oton
and Serra \cite{RS14,RS15a,RS15b}, Grubb 
\cite{G15,G14}. 

For the spectral Dirichlet fractional Laplacian, regularity properties in 
$H^s_p$-spaces have been
known for many years, as a consequence of Seeley's work
\cite{S71,S72}; we shall account for this below in Sections \ref{secSeeley} and \ref{secDir}. 
Further results have recently been presented by Caffarelli and Stinga in
\cite{CS14}, treating domains with limited smoothness
 and obtaining certain H\"older
estimates of Schauder type. See also Cabr\'e{} and Tan \cite{CT10}
Th.\ 1.9, for
the case $a=\frac12$.

In Section \ref{secNeu} we show how similar regularity properties of the spectral Neumann fractional
Laplacian $(-\Delta _{\operatorname{Neu}})^a$ follow from Seeley's
results. Also for this case,
\cite{CS14} has recently  shown H\"older estimates of Schauder type under weaker smoothness
hypotheses. 

In Section \ref{secFurther}, we first briefly discuss extensions to more general scales of function
spaces. Next, for generalizations to nonsmooth domains, we show how
a recent result of Denk, Dore, Hieber, Pr\"uss and Venni
\cite{DDHPV04}, on the existence of $H^\infty $-calculi for boundary
problems, can be combined with more recent results of Yagi
\cite{Y08,Y10}, to extend the regularity properties of Sections \ref{secDir} and \ref{secNeu}
to suitable nonsmooth situations for  small $s$, leading to new results.

Finally, Section \ref{secOver} gives a brief overview of the many kinds of boundary problems
associated with $(-\Delta )^a$, expanding the references given above. This includes several other
Neumann-type problems. 

\medskip
A primary
 purpose of the present note is to put forward some direct
consequences of Seeley \cite{S71,S72} for the spectral fractional
Laplacians. 
One of the main results is that when $A$ is second-order strongly
elliptic and
 $B$ stands for either a Dirichlet or a Neumann
 condition, and $0<a<1$, then for solutions of 
\begin{equation}
(A_B)^au=f,\label{1.2}
\end{equation}
  $f\in H^s_p(\Omega )$ for an $s\ge 0$ implies $u\in
 H^{s+2a}_p(\Omega )$ {\it if and only if} $f$ itself satisfies all those
 boundary conditions of the form $BA^kf=0$ ($k\in {\mathbb N}_0$) that
 have a meaning on $H^s_p(\Omega )$. Consequences are also drawn for
 $C^\infty $-solutions and for solutions where $f$ is in $L_\infty
 (\Omega )$ or a H\"older space. We think this is of interest not just
 as a demonstration of early results, but also in showing how far
 one can reach, as a model for less smooth situations. 
 
Section \ref{secFurther} shows one such generalization to nonsmooth
domains and coefficients.

\section{Seeley's results on complex interpolation}\label{secSeeley}

Let $A$ be a strongly elliptic second-order differential operator on
${\mathbb R}^n$ with $C^\infty $-coefficients. (The following theory
extends readily to $2m$-order systems with normal boundary conditions
as treated in Seeley \cite{S71,S72} and Grubb \cite{G74}, but we restrict the attention to the
second-order scalar case to keep notation and explanations simple.) 

Let $\Omega $ be a $C^\infty $-smooth bounded open subset of ${\mathbb R}^n$, and let  $A_B$ denote the realization of $A$ in
$L_2(\Omega )$ with domain $\{u\in H^2(\Omega )\mid Bu=0\}$; here $Bu=0$
stands for either the Dirichlet condition $\gamma _0u=0$ or a suitable
Neumann-type boundary condition. In details,
\begin{equation}
Bu=\gamma _0B_ju,\text{ where }j=0\text{ or }j=1;\label{2.1}
\end{equation}
here $B_0=I$, and $B_1$ is a first-order differential operator on ${\mathbb R}^n$ such
 that $\{A, \gamma _0B_1\}$ together form a strongly elliptic boundary
 value problem.
Then $A_B$ is
lower bounded with spectrum in a sectorial region $V=\{\lambda
\in{\mathbb C}\mid |\operatorname{Im}\lambda |\le
 C(\operatorname{Re}\lambda -b)\}$. Our considerations in the
 following are formulated for the case where $A_B$ is bijective. Seeley's papers also
 show how to handle a finite-dimensional 0-eigenspace.

The complex powers of $A_B$ can be defined by spectral theory in
$L_2(\Omega )$ in the
cases where $A_B$ is selfadjoint, but Seeley has shown in \cite{S71} how the powers can be defined
more generally in a consistent way, acting in $L_p$-based Sobolev
spaces $H^s_p(\Omega )$ ($1<p<\infty $), by a Cauchy integral of the resolvent around the spectrum
\begin{equation}
(A_B)^z=\tfrac i {2\pi }\int_{\mathcal C}\lambda
^z(A_B-\lambda )^{-1}\,d\lambda .\label{2.2}
\end{equation}
Here $H^s_p({\mathbb R}^n)$ is the set of distributions $u$ (functions if
 $s\ge 0$) such that $(1-\Delta )^{s/2}u\in L_p({\mathbb R}^n)$, and
 $H^s_p(\Omega )=r^+H^s_p({\mathbb R}^n)$ (denoted $\overline
 H^s_p(\Omega )$ in \cite{G14,G15}), where $r^+$ stands for restriction
 to $\Omega $. The general point of view is that the resolvent is
 constructed as an integral operator (found here by pseudodifferential
 methods) that can be applied to various function spaces, e.g.\ when
 $p$ varies. The different realizations coincide on their common
 domains, so the labels $(A_B-\lambda )^{-1}$ and $(A_B)^z$ are used
 without indication of the actual spaces, which are understood from
 the context (this is standard terminology).

 The formula \eqref{2.2} has a good meaning for $\operatorname{Re}z<0$;
 extensions to other values of $z$ are defined by compositions with integer
 powers of $A_B$. As shown in \cite{S71,S72}, one has in general that
 $(A_B)^{z+w}=(A_B)^z(A_B)^w$, and the operators $(A_B)^z$
 consitute a holomorphic semigroup in $L_p(\Omega )$ for
 $\operatorname{Re}z\le 0$.   
This is based on the fundamental estimates of the resolvent shown in
\cite{S69}. For $\operatorname{Re}z>0$, the $(A_B)^z$ define unbounded
 operators in $L_p(\Omega )$, with domains
 $D_p((A_B)^z)=(A_B)^{-z}(L_p(\Omega ))$. Note in particular that
\begin{equation}
 (A_B)^{-z}\colon D_p((A_B)^w)\simto D_p((A_B)^{z+w})\text{ for }\operatorname{Re}z,\operatorname{Re}w>0.\label{2.3}
\end{equation}
We can of course not repeat the full analysis of Seeley here. An abstract
framework for similar constructions of powers of operators in general
Banach spaces is given in Amann \cite{A87,A95}.

The domains in $L_p(\Omega )$ of the positive powers of $A_B$ will now be
explained for the cases $j=0,1$ in (2.1).

The domain of the realization $A_B$ of
$A$ in $L_p(\Omega )$ with boundary condition $Bu=0$ is 
\begin{equation}
D_p(A_B)=\{u\in H_p^2(\Omega )\mid Bu=0\}.\label{2.4}
\end{equation}
In \cite{S72}, Seeley showed that for $0<a<1$,
the domain of $(A_B)^a$
(the range of $(A_B)^{-a}$ applied to $L_p(\Omega )$)
equals the complex interpolation space between 
$L_p(\Omega )$ and $\{u\in H_p^2(\Omega )\mid Bu=0\} $ of the
appropriate order. He showed moreover that this is
the space of functions  $u\in H_p^{2a}(\Omega )$
satisfying $Bu=0$ if $2a>j+\frac1p$, and  the space of functions
$u\in H_p^{2a}(\Omega )$ with no extra condition if
$2a<j+\frac1p$. He gives the special description for the case
$2a=j+\frac1p$:
\begin{equation}
D_p((A_B)^{\frac12(j+\frac1p)})=\{u\in H_p^{j+\frac1p}(\Omega )\mid
B_ju\in \dot H^{\frac1p}_p(\comega ) \};\label{2.5}
\end{equation}
one can say that $B_ju$ vanishes at $\partial\Omega $ in a generalized sense. (It
is also recalled in Triebel [T95], Th.\ 4.3.3.) We here use a notation of
\cite{H85,G14,G15}, where $\dot
H_p^t(\comega)$ stands for
 the space of functions in $H_p^t({\mathbb
R}^n)$ with support in $\comega$. 

Let us define:

\begin{definition}\label{def2.1} The spaces $H_{p,B,A}^s(\Omega )$ are
defined by:
\begin{align}
H_{p,B,A}^s(\Omega )&=H_{p, B}^s(\Omega )=
H_p^s(\Omega )\text{ for }0\le s<j+\tfrac1p,\label{2.6}
\\
H_{p,B,A}^s(\Omega )&=H_{p, B}^s(\Omega )=
\{u\in H_p^s(\Omega )\mid Bu=0\}\text{ for }j< s-\tfrac1p<j+2,\nonumber
\\
H_{p,B,A}^s(\Omega )&=\{u\in H_p^s(\Omega )\mid
Bu=BAu=\dots=BA^ku=0\}\nonumber\\
&\quad\text{ for }j+2k< s-\tfrac1p<j+2(k+1),\nonumber\\
H_{p,B,A}^s(\Omega )&=\{u\in H_p^s(\Omega )\mid
BA^lu=0\text{ for }l<k,\, B_jA^ku\in \dot H_p^{\frac1p}(\comega)\}\nonumber\\
&\quad\text{ when }s-\tfrac1p=j+2k,\nonumber
\end{align}
where $k\in{\mathbb N}_0$.
\end{definition}


Note that in the first three statements, $H_{p,B,A}^s(\Omega )$ consists of the functions
in $H_p^s(\Omega )$ satisfying those boundary conditions
$BA^lu=0$ for which $j+2l<s-\frac1p$ (i.e., those that are well-defined
on $H_p^s(\Omega )$). The definition in the fourth statement, although
slightly complicated, is
included here primarily in order that we can use the notation $H_{p,B,A}^s(\Omega )$
freely without exceptional parameters.

The spaces $H_{p, B}^s(\Omega )$ were defined in Seeley
\cite{S72} (in Grisvard \cite{G67} for $p=2$); we have added the
definitions for $s>2$ (they can be called extrapolation spaces, as in \cite{A87,A95}). In the $L_2$-case, the extra requirement in
\eqref{2.5} can be replaced by $d^{-\frac12}B_ju\in L_2(\Omega )$, where
$d(x)$ is the distance from $x$ to $\partial\Omega $.

With this notation, Seeley's works show:

\begin{theorem}\label{Theorem2.2} 

When $0<a<1$, $D_p((A_B)^a)$ equals the space $[L_p(\Omega ), H_{p,B}^2(\Omega )]_a$
obtained by complex interpolation between  $L_p(\Omega )$ and $
H_{p,B}^2(\Omega )$.

For all $a>0$, $D_p((A_B)^a)=H_{p,
B,A}^{2a}(\Omega )$.
\end{theorem}

\begin{proof} The first statement is a direct quotation from
\cite{S72}. So is the second statement for $0<a\le 1$, and it follows for $a=a'+k$, $0<a'\le 1$ and
$k\in {\mathbb N}$, by using (2.3) with $w=a'$, $z=k$.
\end{proof}
 
Observe the general homeomorphism property that follows from this theorem
in view of formula \eqref{2.3}:

\begin{corollary}\label{Corollary2.3} For $a>0$,
$(A_B)^a$ defines homeomorphisms: 
\begin{equation}
(A_B )^a\colon H^{s+2a}_{p,B ,A}(\Omega )\simto
H^{s}_{p,B ,A}(\Omega ),\text{ for all }s\ge 0.\label{2.7}
\end{equation}
\end{corollary}

The characterization of the interpolation space was given (also for $2m$-order
operators) by Grisvard in
the case of scalar elliptic operators in $L_2$ Sobolev
spaces in \cite{G67}, in terms of real interpolation. Seeley's result
for $1<p<\infty $
is shown for general elliptic operators in vector bundles, with normal boundary conditions.

\section{Consequences for the Dirichlet problem}\label{secDir}

Let $B=\gamma _0$, denoted $\gamma $ for brevity.
Corollary \ref{Corollary2.3} already shows how the regularity
of $u$ and $f=(A_\gamma )^au$ are related, when the functions are known on
beforehand to lie in the special spaces in \eqref{2.6}. But we can also
discuss cases where $f$ is just given in a general Sobolev
space. Namely, we have as a
generalization of the remarks at the end of \cite{S72}:

\begin{theorem}\label{Theorem3.1} Let $0<a<1$. Let $f\in H^{s}_{p}(\Omega )$ for
some $s\ge 0$, and assume that $u\in D_p((A_{\gamma })^a)$ is a solution of
\begin{equation}
(A_{\gamma })^au=f.\label{3.1}
\end{equation}

$1^\circ$ If $s<\frac1p$, then $u\in H^{s+2a}_{p,\gamma }(\Omega )$.

$2^\circ$ Let $\frac1p<s<2+\frac1p$. Then $u\in H^{\frac1p+2a-\varepsilon
}_{p,\gamma }(\Omega )$ for all $\varepsilon >0$. Moreover,  $u\in
H^{s+2a}_{p}(\Omega )$ if and only if $\gamma f=0$, and then in fact $u\in  H^{s+2a}_{p,\gamma
}(\Omega )$.
\end{theorem}

\begin{proof} $1^\circ$. When $s<\frac1p$, we can simply use that
$u=(A_\gamma )^{-a}f$, where $(A_\gamma )^{-a}$ defines a
homeomorphism from $ H^{s}_{p}(\Omega )$ to  $ H^{s+2a}_{p,\gamma
}(\Omega )$ in view of \eqref{2.7}.

$2^\circ$. We first note that since $s>\frac1p>\frac1p-\varepsilon $,
all $\varepsilon >0$, the preceding result shows that $u\in H^{\frac1p+2a-\varepsilon
}_{p,\gamma }(\Omega )$ for all $\varepsilon >0$.

Now if $\gamma f=0$, then $f\in   H^{s}_{p,\gamma
}(\Omega )$ by \eqref{2.6}. Hence  $u\in  H^{s+2a}_{p,\gamma
}(\Omega )$ since  $(A_\gamma )^{-a}$ defines a
homeomorphism from $ H^{s}_{p,\gamma }(\Omega )$ to  $ H^{s+2a}_{p,\gamma
}(\Omega )$ according to \eqref{2.7}.

Conversely, let $u\in  H^{s+2a}_{p}(\Omega )$. Then since we know
already that  $u\in H^{\frac1p+2a-\varepsilon
}_{p,\gamma }(\Omega )$, we see that $\gamma u=0$ (taking $\varepsilon
<2a$). Then by \eqref{2.6}, $u\in H^{\sigma 
}_{p,\gamma }(\Omega )$ for $\frac1p+2a<\sigma <\min\{
s+2a,2+\frac1p\}$; such $\sigma $ exist since $a<1$. Hence  $f\in H^{\sigma -2a
}_{p,\gamma }(\Omega )$ with $\sigma -2a>\frac1p$ and therefore has $\gamma f=0$.
\end{proof}

Point $2^\circ$ in the theorem shows that $f$ may have to be provided with a nontrivial boundary
condition in order for the best possible regularity to hold for
$u$. This is in contrast to the case where
$a=1$, where it is known that for $u$ satisfying $-\Delta u=f$ with
$\gamma u=0$, $f\in H^s_p(\Omega )$ always implies $u\in
H^{s+2}_p(\Omega )$. 

The case $s=\frac 1p$ can be included
in $2^\circ$ when we use the generalized boundary condition in \eqref{2.4};
details are given for the general case in Theorem 3.2 $2^\circ$ below.

The importance of a boundary condition on $f$ for optimal regularity
of $u$ is also demonstrated in the results of
Caffarelli and Stinga \cite{CS14} (and  Cabr\'e{} and Tan
\cite{CT10}).

By induction, we can extend the result to higher $s$:

\begin{theorem}\label{Theorem3.2} Let $0<a<1$. Let $u\in D_p((A_{\gamma })^a)$ be
the solution of {\rm \eqref{3.1}} with $f\in H^s_{p}(\Omega )$ for some $s\ge 0$.
One has for any $k\in {\mathbb N}_0$: 

$1^\circ$ If
$2k+\frac1p<s<2k+2+\frac1p$, and $\gamma A^lf=0$ for $l=0,1,\dots ,k$
(i.e., $f\in H^{s}_{p,\gamma,A }(\Omega )$),
then $u\in H^{s+2a}_{p,\gamma,A }(\Omega )$.

On the other hand, if $u\in H^{s+2a}_{p }(\Omega )$, then necessarily $\gamma
A^lf=0$ for $l=0,1,\dots ,k$ (and hence  $f\in H^{s}_{p,\gamma,A
}(\Omega )$ and $u\in H^{s+2a}_{p,\gamma,A }(\Omega )$). 

$2^\circ$ Let $s=2k+\frac1p$. If $f\in H^{s}_{p,\gamma,A }(\Omega
)$, then $u\in H^{s+2a}_{p,\gamma,A }(\Omega )$. 
On the other hand, if $u\in H^{s+2a}_{p }(\Omega )$, then 
necessarily $f\in  H^{s}_{p,\gamma,A }(\Omega )$ and  $u\in H^{s+2a}_{p,\gamma,A }(\Omega )$. 

\end{theorem}

\begin{proof} Statement $1^\circ$ was shown for $k=0$ in Theorem \ref{Theorem3.1}
$2^\circ$. We proceed by induction: Assume that the statement holds
for $k\le k_0-1$. Now show it for $k_0$:

If $\gamma A^lf=0$ for $l\le k_0$, then $f\in   H^{s}_{p,\gamma,A
}(\Omega )$ by \eqref{2.6}. Hence  $u\in  H^{s+2a}_{p,\gamma,A
}(\Omega )$ since  $(A_\gamma )^{-a}$ defines a
homeomorphism from $ H^{s}_{p,\gamma,A }(\Omega )$ to  $ H^{s+2a}_{p,\gamma,A
}(\Omega )$ according to \eqref{2.7}.

Conversely, let $u\in  H^{s+2a}_{p}(\Omega )$. 
Note that since $s>\frac1p+2k_0>\frac1p+2k_0-\varepsilon $,
all $\varepsilon >0$, the result for $k_0-1$ shows that $u\in H^{\frac1p+2k_0+2a-\varepsilon
}_{p,\gamma,A }(\Omega )$ for all $\varepsilon >0$. Then, taking
$\varepsilon <2a$, we see that $\gamma A^lu=0$ for $l\le k_0$. 
Now in view of \eqref{2.6}, $u\in H^{\sigma 
}_{p,\gamma,A }(\Omega )$ for $\frac1p+2k_0+2a<\sigma <\min\{
s+2a,2+2k_0+\frac1p\}$; such $\sigma $ exist since $a<1$. Hence  $f\in H^{\sigma -2a
}_{p,\gamma,A }(\Omega )$ with $\sigma -2a >2k_0+\frac1p$; therefore it
has $\gamma A^lf=0$ for $l\le k_0$.

The first part of statement $2^\circ$ follows immediately from \eqref{2.7}. For the second part, let  $u\in  H^{s+2a}_{p}(\Omega
)$, $s=2k+\frac1p$. Since $s>2k+\frac1p-\varepsilon $, we see by
application of $1^\circ$ with $s'=2k+\frac1p-\varepsilon $ that  $u\in
H^{2k+\frac1p-\varepsilon +2a}_{p,\gamma ,A}(\Omega )$. For $\varepsilon <2a$
this shows that $\gamma A^lu=0$ for $l\le k$. Now $s+2a=2k+\frac1p+2a$
also lies in $\,]2k+\frac1p, 2k+2+\frac1p[\,$ (since $a<1$) so in fact
$u\in  H^{s+2a}_{p,\gamma ,A}(\Omega )$, and  $f\in  H^{s}_{p,\gamma ,A}(\Omega )$. 
\end{proof}

Briefly expressed, the theorem shows that in order to have
optimal regularity, namely the improvement
from $f$ lying in an  $H^s_p$-space to $u$ lying in an $H^{s+2a}_p$-space,
it is necessary and sufficient to impose all the boundary conditions for the
space $H^s_{p,\gamma,A }(\Omega )$ on $f$.

In the following, we assume throughout that $0<a<1$. (Results for
higher $a$ can be deduced from the present results by use of elementary mapping
properties for integer powers, and are left to the reader.) As a first corollary, we can describe $C^\infty $-solutions.
Define
\begin{equation}
C^\infty _{\gamma,A }(\comega)=\{u\in C^\infty (\comega)\mid \gamma
A^ku=0\text{ for all }k\in{\mathbb N}_0\}.\label{3.2}
\end{equation}

\begin{corollary}\label{Corollary3.3} The operator $(A_\gamma )^a$
defines a homeomorphism of $C^\infty _{\gamma,A }(\comega)$ onto
itself.

Moreover, if $u\in H^{2a}_{p,\gamma,A }(\Omega )\cap C^\infty (\comega)$ for
some $p$, then $(A_\gamma )^au\in C^\infty (\comega)$ implies $u\in
C^\infty _{\gamma,A }(\comega)$ (and hence $(A_\gamma )^au \in C^\infty _{\gamma,A }(\comega)$). 
\end{corollary}

\begin{proof} 
Fix $p$. We first note that 
\begin{equation}
C^\infty _{\gamma,A }(\comega)
=\bigcap_{s\ge 0} H^{s}_{p,\gamma,A }(\Omega ).\label{3.3}
\end{equation}
Here the inclusion '$\subset$' follows from the observation
\[
\{u\in C^\infty (\comega)\mid \gamma
A^lu=0\text{ for }l\le k\} \subset  H^{2k+\frac1p-\varepsilon }_{p,\gamma,A }(\Omega ),
\]
by taking the intersection over all $k$. The other inclusion follows from
$$
H^{2k+\frac1p-\varepsilon }_{p,\gamma,A }(\Omega )\subset \{u\in C^N
(\comega)\mid N<2k+\tfrac1p-\varepsilon -\tfrac np,\,
\gamma
A^lu=0\text{ for }2l\le N\},
$$
by taking intersections for $k\to\infty $.

The fact that $(A_\gamma )^a$ maps $H^{s}_{p,\gamma,A }(\Omega )$
homeomorphically to $H^{s-2a}_{p,\gamma,A }(\Omega )$ for all $s\ge
2a$ now implies that $(A_\gamma )^a$ maps $C^\infty _{\gamma,A
}(\comega)$ to $C^\infty _{\gamma,A }(\comega)$ with inverse
$(A_\gamma )^{-a}$.

Next, let  $u\in H^{2a}_{p,\gamma }(\Omega )\cap C^\infty (\comega)$. If 
$(A_\gamma )^au\in C^\infty (\comega)$, then Theorem 3.2 can be
applied with arbitrarily large $k$, showing that $u\in
C^\infty _{\gamma,A }(\comega)$, and hence $(A_\gamma )^au \in C^\infty _{\gamma,A }(\comega)$.
\end{proof}

\begin{remark}\label{Remark3.4} It follows that for each $1<p<\infty $, the
eigenfunctions of $(A_\gamma )^a$ (with domain $H^{2a}_{p,\gamma }(\Omega
)$) belong to $C^\infty _{\gamma ,A}(\comega)$; they are the same for all
$p$. In particular, when $A_\gamma $ is selfadjoint in
$L_2(\Omega )$, the eigenfunctions of $(A_\gamma )^a$ defined by spectral
theory (that are the same as those of $A_\gamma $) are the eigenfunctions also
in the $L_p$-settings.
\end{remark}

Finally, let us draw some conclusions for regularity properties when
$f\in L_\infty (\Omega )$ or is in a H\"older space.
As in \cite{G15}, we denote by $C^\alpha (\comega)$ the space of functions that are
continuously differentiable up to order $\alpha $ when $\alpha
\in{\mathbb N}_0$, and are in the H\"older class
$C^{k,\sigma }(\comega)$ when $\alpha =k+\sigma $, $k\in{\mathbb N}_0$
and $0<\sigma <1$. 
Recall that the H\"older-Zygmund spaces $B^s _{\infty ,\infty
}(\comega)$, also denoted $C^s_*(\comega)$, coincide with
$C^s(\comega)$ when $s\in{\mathbb R}_+\setminus {\mathbb N}$, and there is
the Sobolev embedding property
$$
H_p^s(\Omega )\subset C^{s-\frac np}_*(\comega)\text{ for all }s>\tfrac np.
$$
(Embedding and trace mapping properties for Besov-Triebel-Lizorkin
spaces $F^s_{p,q}$ and $B^s_{p,q}$ are compiled e.g.\ in Johnsen
\cite{J96}, Sect.\ 2.3, 2.6; note that $H^s_p=F^s_{p,2}$.) Recall also that
$C^k(\comega)\subset C^{k-1,1}(\comega)\subset 
C^k_*(\comega)\subset C^{k-0}(\comega)$ for $k\in{\mathbb N}$. Here we
use the notation $C^{\alpha-0}=\bigcap_{\varepsilon >0}C^{\alpha
-\varepsilon }$ (it is applied 
similarly to $H_p^s$-spaces).

\begin{corollary}\label{Corollary3.5} 
$1^\circ$ 
Let $f\in L_p(\Omega )$ with $\frac n p <2a$. If $2a-\frac n p \ne 1$, 
resp.\ $=1$, then
the solution $u$ of {\rm \eqref{3.1}} is in $C^{2a-\frac np}(\comega)$, resp.\
$C^1_*(\comega)$, with $\gamma
u=0$.

$2^\circ$
If $f\in L_\infty (\Omega )$,
then the solution $u$ of {\rm \eqref{3.1}} is in $C^{2a-0}(\comega)$ with $\gamma u=0$.
\end{corollary}

\begin{proof} $1^\circ$. When $f\in L_p (\Omega )$, then
  $u\in H^{2a}_{p,\gamma }(\Omega)\subset H^{2a}_p(\Omega )$ by
  Theorem 3.1 $1^\circ$. 
Now when $p>\frac n{2a}$, Sobolev embedding gives that $u\in
C^{2a-\frac np}(\comega)$, except when $2a-\frac np=1$, where it gives $u\in
  C^1_*(\comega)$. Since \`a fortiori
 $p>\frac 1{2a}$, we see from \eqref{2.6} that $\gamma u=0$ in
  $H^{2a}_p(\comega)$, hence in $C^{2a-\frac np}(\comega)$ resp.\ $C^1_*(\comega)$.

$2^\circ$. When $f\in L_\infty (\Omega )$, then $f\in L_p (\Omega )$
for all $1<p<\infty $. Using $1^\circ$ and letting $p\to\infty $, we conclude that $u\in C^{2a-0}(\comega)$.
\end{proof}

\begin{corollary}\label{Corollary3.6} Let $k\in{\mathbb N}_0$, and let $2k<\alpha <2k+2$. If $f\in C^\alpha
(\comega )$ with $\gamma A^lf=0$ for $l\le k$, then the solution  $u$
of {\rm \eqref{3.1}} satisfies:
\begin{equation}
u\in C^{\alpha +2a-0}(\comega)\text{ with }\begin{cases} \gamma A^lu=0\text{
for }l\le k \text{ if }\alpha +2a\le 2k+2,\\
\gamma A^lu=0\text{
for }l\le k+1 \text{ if }\alpha +2a> 2k+2.\end{cases} \label{3.4}
\end{equation}

\end{corollary}

\begin{proof}
 When $f\in C^\alpha (\comega)$, then $f\in H^{\alpha
-\varepsilon }_p(\Omega )$ for all $p$, all $\varepsilon >0$. For
$\varepsilon $ so small that $\alpha -\varepsilon >2k$, we see from
\eqref{2.6} that since $\gamma A^lf=0$ for $l\le k$, $f\in H^{\alpha
-\varepsilon }_{p,\gamma,A }(\Omega)$. Then it follows from \eqref{2.7}
that $u\in H^{\alpha +2a-\varepsilon }_{p,\gamma }(\Omega)$. 

If
$\alpha +2a>2k+2$, we have for $\varepsilon $ so small that $\alpha
+2a-\varepsilon >2k+2$, and then $\frac1p$ sufficiently small, that
$u$ satisfies the boundary conditions $\gamma A^lu=0$ for $l\le
k+1$. For $p\to \infty $, this implies that $u\in C^{\alpha
+2a-0}(\comega)$ satisfying these boundary conditions.

If $\alpha +2a\le 2k+2$, we have for $\varepsilon $ in a small
interval $\,]0,\varepsilon _0[\,$ that $2k<\alpha
+2a-\varepsilon <2k+2$, and then for all $p$ sufficiently small, that
$u$ satisfies the boundary conditions $\gamma A^lu=0$ for $l\le
k$. For $p\to \infty $, this implies that $u\in C^{\alpha
+2a-0}(\comega)$ satisfying those boundary conditions.  
\end{proof}

The regularity results of Caffarelli and Stinga \cite{CS14} are
concerned with cases assuming
much less smoothness of the domain and coefficients, getting results
in  H\"older spaces of low order ($<2$). See also Section \ref{secFurther}.

The above results deduced from \cite{S72} explain the role of boundary
conditions on $f$. The results in H\"older spaces
resemble the results of \cite{CS14} for the values of $\alpha
$ considered there, however with a loss of sharpness (the '$-0$') in
some of
the estimates in Corollary \ref{Corollary3.6}.

\section{Consequences for Neumann-type problems}\label{secNeu}

The proofs are analogous for a Neumann-type boundary operator $B$
($j=1$ in \eqref{2.1}ff.).

\begin{theorem}\label{Theorem4.1} Let $0<a<1$. Let $u\in D_p((A_{B })^a)$ be the solution of
\begin{equation}
(A_{B })^au=f,\label{4.1}
\end{equation}
where  $f\in H^{s}_{p}(\Omega )$ for some $s\ge 0$.

$1^\circ$ If $s<1+\frac1p$, then $u\in H^{s+2a}_{p,B }(\Omega )$.

One has for any $k\in {\mathbb N}_0$: 

$2^\circ$ If
$2k+1+\frac1p<s<2k+3+\frac1p$, and $B A^lf=0$ for $l=0,1,\dots ,k$
(i.e., $f\in H^{s}_{p,B,A }(\Omega )$),
then $u\in H^{s+2a}_{p,B,A }(\Omega )$.

On the other hand, if $u\in H^{s+2a}_{p }(\Omega )$, then necessarily $B
A^lf=0$ for $l=0,1,\dots ,k$ (and hence  $f\in H^{s}_{p,B,A
}(\Omega )$ and $u\in H^{s+2a}_{p,B,A }(\Omega )$). 

$3^\circ$ Let $s=2k+1+\frac1p$. If $f\in H^{s}_{p,B,A }(\Omega
)$, then $u\in H^{s+2a}_{p,B,A }(\Omega )$. 
On the other hand, if $u\in H^{s+2a}_{p }(\Omega )$, then 
necessarily $f\in  H^{s}_{p,B,A }(\Omega )$ and  $u\in
H^{s+2a}_{p,B,A }(\Omega )$. 

\end{theorem}

Define
\begin{equation}
C^\infty _{B,A }(\comega)=\{u\in C^\infty (\comega)\mid B
A^ku=0\text{ for all }k\in{\mathbb N}_0\}.\label{4.2}
\end{equation}

\begin{corollary}\label{Corollary4.2} The operator $(A_B )^a$
defines a homeomorphism of $C^\infty _{B,A }(\comega)$ onto
itself.

Moreover, if $u\in H^{2a}_{p,B,A }(\Omega )\cap C^\infty (\comega)$ for
some $p$, then $(A_B )^au\in C^\infty (\comega)$ implies $u\in
C^\infty _{B,A }(\comega)$ (and hence $(A_B )^au \in C^\infty _{B,A }(\comega)$). 
\end{corollary}

\begin{corollary}\label{Corollary4.3}
$1^\circ$ 
Let $f\in L_p(\Omega )$ with $\frac n p <2a$. If $2a-\frac n p \ne 1$, 
resp.\ $=1$, then
the solution $u$ of {\rm \eqref{4.1}} is in $C^{2a-\frac np}(\comega)$, resp.\
$C^1_*(\comega)$, with $B
u=0$ if $2a-\frac n p>1$.

$2^\circ$ If $f\in L_\infty (\Omega )$,
then the solution $u$ of {\rm (4.1)} is in $C^{2a-0}(\comega)$, with
$B u=0$ precisely when $a>\frac12$.
\end{corollary}

\begin{proof} $1^\circ$ It is seen as in Corollary \ref{Corollary3.5} that $u\in
C^{2a-\frac np}(\comega)$ resp.\ $C^1_*(\comega)$. If $2a-\frac n p>1$, then
\`a fortiori $2a-\frac 1 p>1$, and $Bu=0$ in $H^{2a}_p(\Omega )$; this
carries over to the space we embed in.

 $2^\circ$. When $f\in L_\infty (\Omega )$, then $f\in L_p (\Omega )$
for all $1<p<\infty $, so we have $1^\circ$ for all $p$. 
Letting $p\to\infty
$, we conclude that $u\in C^{2a-0}(\comega)$, and $Bu=0$ is assured if
$2a>1$.
When  $a\le \frac12$, then $2a\le 1<1+\frac1p$ for all $p$, so $H^{2a}_{p,B }(\Omega)=H^{2a}_{p}(\Omega)$ for all
$p$; no boundary condition is imposed. 
\end{proof}

\begin{corollary}\label{Corollary4.4} Let $k\in{\mathbb N}_0$, and let $\alpha \ge 0$
satisfy $2k-1<\alpha
<2k+1$.

If $f\in C^\alpha
(\comega )$ with $B A^lf=0$ for $l\le k-1$, then the solution  $u$
of {\rm \eqref{4.1}} satisfies:
\begin{equation}
u\in C^{\alpha +2a-0}(\comega)\text{ with }\begin{cases} B A^lu=0\text{
for }l\le k -1\text{ if }\alpha +2a\le 2k+1,\\
B A^lu=0\text{
for }l\le k\text{ if }\alpha +2a> 2k+1.\end{cases} \label{4.3}
\end{equation}

\end{corollary}

In the case of $(-\Delta _{\operatorname{Neu}})^a$ considered on a
connected set $\Omega $, there is a
one-dimensional nullspace consisting of the constants (that are of
course in $C^\infty (\comega)$). This case is included in the above
results by a trick found in \cite{S71}: Replace $-\Delta $ by 
\begin{equation}
A=-\Delta +E_0, \quad E_0u= \frac 1{\operatorname{vol}(\Omega )}\int_{\Omega }u(x)\,dx;\label{4.4}
\end{equation}
note that $E_0$ is a projection onto the constants, orthogonal in
$L_2(\Omega )$ (it is also a pseudodifferential operator of order $-\infty
$). Here $\Delta E_0=0$ and $\gamma _1E_0=0$, where $
\gamma
_1u=\partial _nu|_{\partial\Omega }$.  With $B=\gamma _1$,
 $(A_{\gamma _1})^a$ equals $(-\Delta
_{\gamma _1})^a+E_0$ and is invertible, and the above results apply to it
and lead to similar regularity results for $(-\Delta _{\gamma _1})^a$
itself
(note that $\gamma _1A^ku=\gamma _1(-\Delta )^ku$).

\section{Further developments}\label{secFurther}

\subsection{More general function spaces}

The above theorems in $L_p$ Sobolev spaces are likely to extend to a large number of other
scales of function spaces. Notably, it seems possible to extend them
to the scale of Besov spaces $B^s_{p,q}$ with
$1\le p\le \infty $,  $1\le q< \infty $, since the decisive complex
interpolation properties of domains of elliptic realizations have been
shown by Guidetti in [G91]. 

It is not at the moment clear to the author whether the scale $B^s_{\infty
,\infty }=C^s_*$
of 
 H\"older-Zygmund spaces, or the scale of ``small'' H\"older-Zygmund
 spaces $c^s_*$ (obtained by closure in $C^s_*$-spaces of the compactly
 supported smooth functions), cf.\ e.g.\ Escher and Seiler
 \cite{ES08}, can be or has been included for these
 boundary value problems. (It was possible to include $C^s_*$ in the regularity
study for the restricted fractional Laplacian in \cite{G14} using
Johnsen \cite{J96}.) Such an extension
would allow removing the '$-0$' in some formulas in Corollaries \ref{Corollary3.6}
and \ref{Corollary4.4} above.

Let us mention for cases {\it without} boundary conditions, that the continuity of classical pseudodifferential operators
on ${\mathbb R}^n$ (such as $(-\Delta )^a$ and its parametrices) in H\"older-Zygmund spaces has been known for many years, cf.\ e.g.\
Yamazaki \cite{Y86} for a more general result and references to
earlier contributions. On this point, \cite{CS14} refers to Caffarelli
and Silvestre \cite{CS07}.

\subsection{Nonsmooth situations}

It is of great interest to treat the problems also when the set
$\Omega $ and the coefficients of $A$ have only limited smoothness.
One of the common strategies is to transfer the results known for
constant-coefficient operators on $\rnp$ to to variable-coefficient operators by perturbation
arguments, and to sets $\Omega $ by local
coordinates. (This strategy is used in \cite{CS14}.)
The pseudodifferential theory in smooth cases is in fact set up to
incorporate the perturbation arguments in a systematic and more
informative calculus. For nonsmooth cases,
we remark that there do exist
pseudodifferential theories requiring only limited smoothness in $x
$, cf.\
\cite{AGW14} and other works of Abels listed there. Applications to
the present problems await development.

Another point of view comes forward in the efforts to establish
so-called maximal regularity, $H^\infty $-calculus and  $R$-boundedness
properties for operators generating semigroups; see e.g.\ Denk,
Hieber and Pr\"uss \cite{DHP03} for results, references to the vast literature, and an
overview of the theory. Fractional powers of boundary problems entered
in this theory at an early stage, starting with Seeley's results, but
are not so much in focus in the latest developments, that are
primarily aimed towards solvability of parabolic problems. 

However, there is
an interesting result by Yagi \cite{Y08} that is relevant for the present
purposes. He considers an operator
\begin{equation}
A=-\sum_{j,k=1,\dots,n}\partial_ja_{jk}(x)\partial_k+c(x),\text{ with
}\sum_{j,k=1,\dots,n}a_{jk}(x)\xi _j\xi _j\ge c_0|\xi |^2 ,\label{5.1}
\end{equation}
 $a_{jk}=a_{kj}$ real in $C^1(\comega)$, $c(x)$ real bounded $\ge 0$ and $c_0>0$,
on a
bounded $C^2$-domain $\Omega \subset \rn$.  
Define
\begin{equation}
 H^s_{p,\gamma }(\Omega
)= \begin{cases} H^s_p(\Omega )\text{ for }0\le s<\tfrac1p,\\
 \{u\in H^s_p(\Omega )\mid \gamma u=0\}\text{ for }\tfrac1p<s\le 2.\end{cases}\label{5.2}
\end{equation}
Since
$A=-\sum_{j,k}(a_{jk}\partial_j\partial_k+(\partial_ja_{jk})\partial_k)+c$
with $a_{jk}\in C^1$ and  $\partial a_{jk}\in C^0$, it follows
from Denk, Dore, Hieber, Pr\"uss and Venni \cite{DDHPV04} Th.\ 2.3, for $1<p<\infty $, that the Dirichlet
realization $A_{\gamma }$ of $A$ in $L_p(\Omega )$ with domain 
\begin{equation*}
D_p(A_{\gamma })=H^2_{p,\gamma }(\Omega ),
\end{equation*}
admits a bounded $H^\infty $-calculus in $L_p(\Omega )$. We here use
that for $p=2$, $A_{\gamma }$ is selfadjoint in $L_2(\Omega )$ with a positive lower
bound (since $\Omega $ is bounded), hence the constant $\mu 
_\phi $ in the theorem can be taken equal to 0.
We also observe that the definitions of the operators for various $p$ are consistent
(and they all have the same eigenvector system).

Combined with the existence of an
$H^\infty $-calculus, Theorem 5.2 of \cite{Y08} then shows:

\begin{theorem}\label{Theorem5.1} Let $1<p<\infty $. For $0\le a\le 1$, the
fractional  powers
$(A_{\gamma })^a$ in $L_p(\Omega )$ have domains
\begin{equation}
D_p((A_{\gamma })^a)=\begin{cases} H^{2a}_p(\Omega ) \text{ if }0\le 2a<\tfrac1p,\\ 
H^{2a}_{p,\gamma }(\Omega ) \text{ if }\tfrac1p<2a\le 2,\; 2a\ne
1+\tfrac1p.
\end{cases} \label{5.3}
\end{equation}

\end{theorem}

(\cite{Y08} does not describe the excepted cases $s=\frac1p,1+\frac1p$.) 

With this statement we
can  repeat the proof of Theorem \ref{Theorem3.1} in cases where $s\le 2-2a$,
obtaining:

\begin{theorem}\label{Theorem5.2} (Recall the smoothness assumptions: $\Omega $ is $C^2$ and the
$a_{jk}$ are in $C^{1}(\comega)$, $c\in L_\infty (\Omega )$.)

Let $0<a<1$. Let $f\in H^{s}_{p}(\Omega )$ for
some $s\in [0, 2-2a]$, and assume that $u\in D_p((A_{\gamma })^a)$ is a solution of
\begin{equation}
(A_{\gamma })^au=f.\label{5.4}
\end{equation}
Assume that $s$ and $s+2a$ are different from $\frac1p$ and $1+\frac1p$.

$1^\circ$ If $s<\frac1p$, then $u\in H^{s+2a}_{p,\gamma }(\Omega )$.

$2^\circ$ Let $\frac1p<s\le 2-2a$. Then $u\in H^{\frac1p+2a-\varepsilon
}_{p,\gamma }(\Omega )$ for all $\varepsilon >0$. Moreover,  $u\in
H^{s+2a}_{p}(\Omega )$ if and only if $\gamma f=0$, and then in fact $u\in  H^{s+2a}_{p,\gamma
}(\Omega )$.
\end{theorem}

\begin{proof} 
We first note that by the general properties of fractional powers,
\begin{equation}
(A_{\gamma })^a\colon D_p((A_{\gamma })^{t+a})\simto D_p((A_{\gamma
})^{t}), \text{ for }t\ge 0;\label{5.5}
\end{equation}
this covers part of the statements in view of Theorem \ref{Theorem5.1}.

$1^\circ$ follows from \eqref{5.5}, since $H^s_p(\Omega )=D_p((A_{\gamma })^{s/2})$
for $s<\frac1p$ and $D_p((A_{\gamma })^{s/2+a})=H^{s+2a}_{p,\gamma
}(\Omega )$, by \eqref{5.3}.

For $2^\circ$, we first note that since $s>\frac1p>\frac1p-\varepsilon $,
all $\varepsilon >0$, the preceding result shows that $u\in H^{\frac1p+2a-\varepsilon
}_{p,\gamma }(\Omega )$ for all $\varepsilon >0$.

Now if $\gamma f=0$, then $f\in   H^{s}_{p,\gamma
}(\Omega )$ by \eqref{5.1}, which equals $D_p((A_{\gamma })^{s/2})$ by
\eqref{5.3}, and hence  $u\in D_p((A_{\gamma })^{s/2+a})= H^{s+2a}_{p,\gamma
}(\Omega )$ in view of \eqref{5.5} and \eqref{5.3}.

Conversely, let $u\in  H^{s+2a}_{p}(\Omega )$. Then since we know
already that  $u\in H^{\frac1p+2a-\varepsilon
}_{p,\gamma }(\Omega )$, we see that $\gamma u=0$ (taking $\varepsilon
<2a$). Then by \eqref{5.3}, $u\in H^{\sigma 
}_{p,\gamma }(\Omega )$ for $\frac1p+2a<\sigma <\min\{
s+2a,2+\frac1p\}$; such $\sigma $ exist since $a<1$. Hence  $f\in H^{\sigma -2a
}_{p,\gamma }(\Omega )$ with $\sigma -2a>\frac1p$ and therefore has $\gamma f=0$.
\end{proof}

Case $2^\circ$ is of course only relevant when $a<1-\frac 1{2p}$.

Now one can draw corollaries exactly as in Corollaries \ref{Corollary3.5} and \ref{Corollary3.6}:

\begin{corollary}\label{Corollary5.3} Let $u$ be a solution of {\rm \eqref{5.4}}.

$1^\circ$  
Let $f\in L_p(\Omega )$ with $\frac n p <2a$, $2a\notin\{\frac1p, 1+\frac1p\}$. If $2a-\frac n p \ne 1$, 
resp.\ $=1$, then
the solution $u$ of {\rm (2.1)} is in $C^{2a-\frac np}(\comega)$, resp.\
$C^1_*(\comega)$, with $\gamma
u=0$.

$2^\circ$ If $f\in L_\infty (\Omega )$, then $u\in C^{2a-0}(\comega)$
with $\gamma u=0$.

$3^\circ$ If  $f\in C^\alpha (\comega)$ with $\gamma f=0$ for some
$\alpha \in \,]0, 2-2a]$, then $u\in C^{\alpha +2a-0}(\comega)$
with $\gamma u=0$.
\end{corollary}

In the cases $2a=\frac1p$ or $1+\frac1p$ in $1^\circ$, one has at least
that $u\in H^{2a-0}_{p,\gamma }(\Omega )$, from which one concludes
$u\in C^{2a-\frac np-0}(\comega)$ with $\gamma u=0$.

It would be natural to generalize the results of Yagi to boundary problems
for higher-order
operators $A$, including integer powers of $A_\gamma $ (the latter
would make it possible to consider larger $s$ in Theorem \ref{Theorem5.2} under
increased smoothness requirements), but to
our knowledge, no
such efforts seem to have been made so far.

In  the book of Yagi \cite{Y10}, Chapter 16, there are shown
similar results for the Neumann problem; here $c(x)\ge c_1>0$ in \eqref{5.1}
and the boundary operator is the conormal derivative  
\begin{equation*}
Bu=\sum_{j,k=1,\dots,n}\nu _j\gamma (a_{jk}\partial_ku),
\end{equation*}
where $\nu =\{\nu _1,\dots,\nu _n\}$ is the normal to $\partial\Omega $.
We define
\begin{equation}
H^s_{p,B }(\Omega
)= \begin{cases} H^s_p(\Omega )\text{ for }0\le s<1+\tfrac1p,\\
 \{u\in H^s_p(\Omega )\mid B u=0\}\text{ for }1+\tfrac1p<s\le 2.\end{cases}\label{5.6}
\end{equation}
It follows from \cite{DDHPV04} Th.\ 2.3, for $1<p<\infty $, that the Neumann
realization $A_{B }$ of $A$ in $L_p(\Omega )$ with domain $
D_p(A_{B })=H^2_{p,B }(\Omega )$
admits a bounded $H^\infty $-calculus in $L_p(\Omega )$. 
Then Th.\ 16.11 of \cite{Y10} implies that  the fractional powers
$(A_{B })^a$ in $L_p(\Omega )$ for $0<a<1$ have domains
\begin{equation}
D_p((A_{B})^a)=\begin{cases} H^{2a}_p(\Omega ) \text{ if }0\le 2a<1+\tfrac1p,\\ 
H^{2a}_{p,B}(\Omega ) \text{ if }1+\tfrac1p<2a\le 2.
\end{cases} \label{5.7}
\end{equation}
We can now extend the results in Section \ref{secDir} to this nonsmooth
situation, when $s\le 2-2a$, $\alpha \le 2-2a$. The proofs are the
same as there, only
used in the applicable range.

\begin{theorem}\label{Theorem5.4} 
Let $0<a<1$. Let $f\in H^{s}_{p}(\Omega )$ for
some $s\in [0, 2-2a]$, and assume that $u\in D_p((A_{B })^a)$ is a solution of
\begin{equation}
(A_{B })^au=f.\label{5.8}
\end{equation}
Assume that $s$ and $s+2a$ are different from $1+\frac1p$.

$1^\circ$ If $s<1+\frac1p$, then $u\in H^{s+2a}_{p,B }(\Omega )$.

$2^\circ$ Let $1+\frac1p<s\le 2-2a$. Then $u\in H^{1+\frac1p+2a-\varepsilon
}_{p,B }(\Omega )$ for all $\varepsilon >0$. Moreover,  $u\in
H^{s+2a}_{p}(\Omega )$ if and only if $Bf=0$, and then in fact $u\in  H^{s+2a}_{p,B
}(\Omega )$.
\end{theorem}

Here $2^\circ$ is only relevant when $a<\frac12 - \frac1{2p}$.

\begin{corollary}\label{Corollary5.5} Let $u$ be a solution of {\rm \eqref{5.8}}.

$1^\circ$ 
Let $f\in L_p(\Omega )$ with $\frac n p <2a$, $2a\ne 1+\frac1p$. If $2a-\frac n p \ne 1$, 
resp.\ $=1$, then
the solution $u$ of {\rm \eqref{5.8}} is in $C^{2a-\frac np}(\comega)$, resp.\
$C^1_*(\comega)$, with $B
u=0$ if $2a-\frac n p>1$.

$2^\circ$ If $f\in L_\infty (\Omega )$, then $u\in C^{2a-0}(\comega)$,
with $B u=0$ precisely when  $a>\frac12$.

$3^\circ$ If  $f\in C^\alpha (\comega)$ 
for some
$\alpha \in \,]0, 2-2a]$, with $B f=0$ if $\alpha >1$, then $u\in C^{\alpha +2a-0}(\comega)$,
with $B u=0$ if $\alpha +2a>1$.
\end{corollary}

For $2a=1+\frac1p$ in $1^\circ$ one gets $C^{2a-\frac np-0}(\comega)$
instead of $C^{2a-\frac np}(\comega)$.

\begin{remark}\label{Remark5.6} Whereas the results in Theorems \ref{Theorem5.2} and \ref{Theorem5.4} for general
$s$ are new, those in $1^\circ$ and $3^\circ$ of Corollaries \ref{Corollary5.3} and \ref{Corollary5.5} are
comparable to the results of Caffarelli and Stinga \cite{CS14}. The smoothness assumptions there
are up to 1 step weaker than ours. 
On the other hand, for $1^\circ$, the case  $2a=\frac np$ is not
addressed in \cite{CS14}, and the validity of the boundary conditions in
the standard sense for $u$ is not discussed. For $3^\circ$, our result
misses the best H\"older space for $u$ by an $\varepsilon $, but we
treat $f$ in the full range $\alpha\le 2-2a$, not assuming $\alpha <1$
on beforehand.
\end{remark}

One can moreover deduce results
in $L_2$ Sobolev spaces for more rough domains (Lipschitz or convex) from
\cite{Y10}.

\section{Overview of boundary problems associated with the fractional
 Laplacian}\label{secOver}

For the convenience of the reader, 
we here go through various boundary value problems associated with $(-\Delta
)^a$, $0<a<1$. For the problems considered in Sections 6.1 and 6.2, one can consider generalizations where $-\Delta $ is
replaced by a variable-coefficient second-order strongly elliptic differential
operator. More generally, one can replace $(-\Delta )^a$ by an
elliptic pseudodifferential operator $P$ of order $2a$ having the
so-called $a$-transmission property at $\partial\Omega $, cf.\
\cite{G15,G14}.

In much of the recent
literature,   $(-\Delta )^a$ is presented in the form
\begin{equation}
(-\Delta )^au(x)=c_{n,a}\operatorname{PV}\int_{{\mathbb R}^n}\frac{u(x)-u(y)}{|y|^{n+2a}}\, dy.\label{6.1}
\end{equation}
This is sometimes generalized by replacing $|y|^{-n-2a}$ by other
nonnegative functions
$K(y)$, satisfying $K(-y)=K(y)$ and  homogeneous of degree $-n-2a$.
(Cf.\ e.g.\ \cite{RS15a,RS15b} and their references; in the case where
$K$ is $C^\infty $ outside 0, this defines an operator of the type $P$ mentioned above.) More generally, $K$ can be 
subject to estimates comparing with $|y|^{-n-2a}$.

\subsection{The restricted Dirichlet and Neumann fractional
Laplacians}\label{resfrac} 

The 
 properties of the restricted Dirichlet fractional Laplacian $(-\Delta
)^a_{\operatorname{Dir}}$ defined in the introduction were studied
e.g.\ in
Blumenthal and Getoor \cite{BG59}, Landkof \cite{L72}, Hoh and Jacob
\cite{HJ96},  Kulczycki
\cite{K97}, Chen and Song \cite{CS98}, Jakubowski \cite{J02},
Silvestre \cite{S07},
Caffarelli and Silvestre \cite{CS09}, Frank and Geisinger
\cite{FG14}, 
Ros-Oton and Serra \cite{RS14,RS15a}, Felsinger, Kassmann and Voigt
\cite{FKV14}, Grubb
\cite{G14,G15},  Bonforte, Sire
and Vazquez \cite{BSV14}, Servadei and Valdinoci \cite{SV14}, Binlin,
Molica Bisci and Servadei \cite{BMS15}, and many
more papers referred to in these works (see in particular the list in
\cite{SV14}).

The operator acts like $r^+(-\Delta )^a$ applied to functions
supported in $\comega$.
The domain in $L_2(\Omega )$ is for $a<\frac12$  equal to $\dot
H^{2a}_2(\comega)$ (the $H^{2a}_2({\mathbb R}^n)$-functions supported in
$\comega $), and has for $a\ge \frac12$ been described in exact form in \cite{G14,G15} by
\begin{equation}
D_2((-\Delta )^a_{\operatorname{Dir}})=H_2^{a(2a)}(\comega)=\Lambda
_+^{(-a)}e^+\overline H^a_2(\Omega ).\label{6.2}
\end{equation}
 Here $\Lambda _+^{(\mu )}$ is a so-called order-reducing operator of
 order $\mu \in\mathbb C$ that preserves support in $\comega$, $e^+$
 extends by zero on ${\mathbb R}^n\setminus\Omega $, and $\overline
 H^s_p(\Omega )$ is the sharper notation for $H^s_p(\Omega )$ used in
 \cite{G14,G15}. H\"ormander's spaces $H_p^{\mu (s)}(\comega)$ are defined there
 in general by
\begin{equation}
H_p^{\mu (s)}(\comega)=\Lambda _+^{(-\mu )}e^+\overline
H^{s-\operatorname{Re}\mu }_p(\Omega ), \text{ for
}s-\operatorname{Re}\mu >-1+1/p.\label{6.3} 
\end{equation}

The operator $(-\Delta )^a_{\operatorname{Dir}}$ represents the {\it homogeneous}
Dirichlet problem, and there is an associated well-posed {\it
nonhomogeneous Dirichlet problem} defined on a larger space:
\begin{equation}
\begin{cases} r^+(-\Delta )^au& =f\text{ on }\Omega ,\\
\operatorname{supp}u&\subset \comega ,\\
\gamma _{a-1,0}u &=\varphi \text{ on }\partial\Omega  ,
\end{cases} \label{6.4}
\end{equation}
where $\gamma _{a-1,0}u=c_0(d^{1-a}u)|_{\partial\Omega }$ with
$d(x)=\operatorname{dist}(x,\partial\Omega )$. When $f\in
H^{s-2a}_p(\Omega )$, the solutions are in
spaces $H_p^{(a-1)(s)}(\comega)$, which allow a blowup of $u$ (of the
form $d^{a-1}$) at
$\partial\Omega $, see \cite{G14,G15} and also Abatangelo \cite{A14}. The solutions with
$\varphi =0$ are exactly the solutions of the homogeneous Dirichlet
problem, lying in  $H_p^{a(s)}(\comega)$ and behaving like $d^a$ at the
boundary. 

Likewise, one can define a well-posed {\it  nonhomogeneous Neumann
problem} (cf.\ \cite{G14})
\begin{equation}
\begin{cases} r^+(-\Delta )^au& =f\text{ on }\Omega ,\\
\operatorname{supp}u&\subset \comega ,\\
\gamma _{a-1,1}u
 &=\psi  \text{ on }\partial\Omega ,
\end{cases} \label{6.5} 
\end{equation}
where $\gamma _{a-1,1}u=c_1\partial_n(d(x)^{1-a}u)|_{\partial\Omega }$; it  has
solutions in $H_p^{(a-1)(s)}(\comega)$. There is then a {\it homogeneous
Neumann problem}, with $\psi =0$ in \eqref{6.5}; its solutions for $f\in
H^{s-2a}_p(\Omega )$ lie in a closed
subset of $H_p^{(a-1)(s)}(\comega)$.

These boundary conditions are {\it local}; one can also impose {\it nonlocal}
pseudodifferential boundary conditions prescribing $\gamma _0Bu$ with
a pseudodifferential operator $B$, see \cite{G14}, Section 4A.

The problems \eqref{6.4} and \eqref{6.5} are sometimes considered with the
condition $\operatorname{supp}u\subset \comega$ replaced by
prescription of
a nontrivial value $g$ of $u$ on ${\mathbb R}^n\setminus\comega$. It is
accounted for e.g.\ in \cite{G14} how such problems are reduced to the
case where $g=0$ as in \eqref{6.4}, \eqref{6.5}.

\subsection{The spectral Dirichlet and Neumann fractional Laplacians}
\label{specfrac}

Fractional powers of  realizations of the Laplacian and other elliptic
operators have been considered for many
years. In the case of a selfadjoint operator in $L_2(\Omega )$, there
is an operator-theoretical definition by spectral theory. More general,
not necessarily selfadjoint cases can be included,
when the powers are defined by a Dunford integral as in
\eqref{2.2}. Moreover, this representation allows a discussion of the
analytical structure. The
structure of powers of differential operators acting on a 
manifold without boundary, was cleared up
by Seeley \cite{S67}, who showed that they are classical
pseudodifferential operators. The case of realizations $A_B$ on a
manifold with boundary was described by Seeley in \cite{S71,S72},
based on \cite{S69}. The resulting operators $(A_B)^a$ have been further
analyzed in the book \cite{G96}, Section 4.4, from which follows that they are
sums of a truncated pseudodifferential term $r^+A^ae^+$ and a generalized singular
Green operator, having its importance at the boundary; here
$e^+$ denotes extension by zero (on ${\mathbb R}^n\setminus\Omega $). 
 (The detailed
analysis of the singular Green term is complicated.)
Fractional powers are of interest in differential geometry e.g.\ for the determination of topological
constants such as residues or indices.

The operators 
 have been considered more recently for questions arising in nonlinear PDE. 
Stinga and Torrea \cite{ST10}, Cabr\'e and Tan \cite{CT10} for $a=\frac12$, and
Caffarelli and Stinga \cite{CS14} for both $(-\Delta
_{\operatorname{Dir}})^a$ and $(-\Delta _{\operatorname{Neu}})^a$,
show how the spectral 
fractional Laplacians can be defined on a bounded domain by a
generalization of the 
Caffarelli-Silvestre extension \cite{CS07} to cylindrical situations.
The paper of Servadei and
 Valdinoci \cite{SV14}, which compares the eigenvalues of $(-\Delta
 _{\operatorname{Dir}})^a$ and $(-\Delta )^a _{\operatorname{Dir}}$,
 contains an extensive list of  references  to the recent literature,
 to which we 
 refer. See also Bonforte, Sire and Vazquez \cite{BSV14}, Capella,
 Davila, Dupaigne and Sire \cite{CDDS11}, and their references.

The regularity analyses of \cite{CT10,CS14} were preceded by that of
\cite{S71,S72} accounted for above.

It should be noted that the operators $(-\Delta )^a_{\operatorname{Dir}}$ and  $(-\Delta
_{\operatorname{Dir}})^a$ are both selfadjoint positive in $L_2(\Omega )$, but they
act differently, and their domains differ when $a\ge \frac12$.

For the spectral Dirichlet and Neumann fractional Laplacians there have
not been formulated nonhomogeneous boundary problems. In constrast,
the restricted  Dirichlet and Neumann fractional Laplacians  allow
nonhomogeneous boundary conditions.


\subsection{Two other Neumann cases} \label{otherNeu}

For completeness, we moreover mention two further choices of operators
associated with the fractional Laplacian and a set $\Omega $, namely operators defined 
from the sesquilinear forms
\begin{align}
p_0(u,v)&=\tfrac12 {c_{n,a}}\int_{\Omega \times\Omega
}\frac{(u(x)-u(y))(\bar v(x)-\bar v(y))}{|x-y|^{n+2a}}\, dxdy,\label{6.6}\\
p_1(u,v)&=\tfrac12 {c_{n,a}}\int_{{\mathbb R}^{2n}\setminus(\complement\Omega \times\complement\Omega
)}\frac{(u(x)-u(y))(\bar v(x)-\bar v(y))}{|x-y|^{n+2a}}\, dxdy.\nonumber
\end{align}

It is known that $(p_0(u,u)+\|u\|^2)^\frac12$ is equivalent with the norm
on $H^a_2(\Omega )$.
By a variational construction, $p_0$ with domain $H_2^a(\Omega )$ gives
rise to a selfadjoint operator $P_0$ in $L_2(\Omega )$, sometimes called ``the regional fractional
Laplacian''. To see how it acts, we note that one has from \eqref{6.1}, for suitable
functions $U,V$ on ${\mathbb R}^n$, 
\begin{equation*}
((-\Delta )^aU,V)_{{\mathbb R}^n}=\tfrac12 {c_{n,a}}\int_{{\mathbb R}^{2n}
}\frac{(U(x)-U(y))(\overline V(x)-\overline V(y))}{|x-y|^{n+2a}}\, dxdy
\end{equation*}
 (the factor $\frac12$ comes in since $V$ appears twice); hence for
 $u,v$ given on $\Omega $,
\begin{align*}
(&(-\Delta )^ae^+u,e^+v)_{{\mathbb R}^n}=\tfrac12 {c_{n,a}}\int_{{\mathbb R}^{2n}
}\frac{(e^+u(x)-e^+u(y))(e^+\bar v(x)-e^+\bar v(y))}{|x-y|^{n+2a}}\, dxdy\\
&=p_0(u,v)+\tfrac12 {c_{n,a}}\int_{x\in\Omega ,y\in\complement\Omega 
}\frac{u(x)\bar v(x)}{|x-y|^{n+2a}}\, dxdy+\tfrac12 {c_{n,a}}\int_{y\in\Omega ,x\in\complement\Omega 
}\frac{u(y)\bar v(y)}{|x-y|^{n+2a}}\, dxdy\\
&=p_0(u,v)+(wu,v)_{\Omega },\text{ where } w(x)=
c_{n,a}\int_{y\in\complement\Omega 
}\frac{1}{|x-y|^{n+2a}}\, dxdy.
\end{align*}
It follows that the operator $P_0$
acts like $u\mapsto r^+(-\Delta )^ae^+u-wu$; observe that the function
$w$ has a singularity at
$\partial\Omega $ (balancing the singularity of the first term).
This case appears e.g.\ in Lieb and Yau \cite{LY88}, 
Chen and Kim  \cite{CK02}, Bogdan, Burdzy and Chen
\cite{BBC03}. For
$\frac12<a<1$, it is shown in Guan \cite{G06} how $P_0$
represents a  Neumann condition $(d^{2-2a}\partial_nu)|_{\partial\Omega }=0$. Nonhomogeneous Neumann and Robin problems for the
regional fractional Laplacian are studied
in Warma \cite{W15}.

The other choice $p_1$ has recently been introduced in Dipierro, Ros-Oton and
Valdinoci  in \cite{DRV14} (formulated for real functions), where it is shown how it defines an
operator $r^+(-\Delta )^a$ applied to functions on ${\mathbb
R}^n$ satisfying a special condition viewed as a ``nonlocal Neumann
condition'', relating the behavior in ${\mathbb R}^n\setminus
\Omega $ to that in $\Omega $. Here one can also define  nonhomogeneous
nonlocal Neumann problems.

\begin{acknowledgement}
The author is grateful to H.\ Abels, J.\ Johnsen, X.\ Ros-Oton and A.\ Yagi for
useful discussions.
\end{acknowledgement}

\bibliographystyle{mn}

\providecommand{\WileyBibTextsc}{}
\let\textsc\WileyBibTextsc
\providecommand{\othercit}{}
\providecommand{\jr}[1]{#1}
\providecommand{\etal}{~et~al.}

\end{document}